\documentclass[10pt,a4paper]{amsart}
\usepackage{amsmath, amssymb, latexsym, amscd, amssymb}
\usepackage{changepage}
\usepackage{pstricks,pst-plot,pst-node,pst-grad,pst-3d, hhline, multirow, graphicx,caption,subcaption}
\newtheorem{thm}{Theorem}[section]
\newtheorem{theorem}{Theorem}[section]
\newtheorem{cor}[thm]{Corollary}

\newtheorem{lemma}[thm]{Lemma}

\newtheorem{proposition}[thm]{Proposition}
\newtheorem{example}{Example}[section]
\theoremstyle{definition}

\theoremstyle{remark}

\newtheorem{remark}[thm]{Remark}
\numberwithin{equation}{section}

 \def\bn{{\mathbf n}}

\def\bw{{\mathbf w}} \def\bz{{\mathbf z}}
\def\bw{{\mathbf w}}

\def\pd#1#2{\frac{\partial #1}{\partial#2}}

 \def\eps{\epsilon}

\newcommand{\bd}{\partial}

\newcommand{\wtilde}{\widetilde}
\newcommand{\what}{\widehat}

\newcommand{\tnorm}[1]{|\!|\!|#1|\!|\!|}

\newcommand{\bt}{{\mathbf{t}}}

\newcommand{\T}{{\mathcal{T}}}

\usepackage{colortbl}

\begin{document}

\title[ A modified $P_1$ - immersed Finite element method ]
{ A modified $P_1$ - immersed finite element method}

\author[  D. Y. Kwak ]
{ Do Y. Kwak  }
\address{ Department of Mathematical Sciences, KAIST,\\  291 Daehak-ro, Yuseong-gu,
Daejeon, Korea 305-701.    }
\email{kdy@kaist.ac.kr}
\thanks{This author was supported by the National Research Foundation of Korea, No.2014R1A2A1A11053889}
 \author[ Juho Lee ]
{ Juho Lee }
\address{Department of Mathematical Sciences, KAIST, \\  291 Daehak-ro, Yuseong-gu,
Daejeon, Korea 305-701. }
\email{cool295@kaist.ac.kr}

%\date{10 May}

\begin{abstract}
In recent years, the immersed finite element methods (IFEM) introduced in \cite{Li2003}, \cite{Li2004} to solve elliptic problems having an interface in the domain due to the discontinuity of coefficients are getting more attentions of researchers because of their simplicity and efficiency.
Unlike the conventional finite element methods, the IFEM allows the interface to cut through the interior of the element, yet
after the basis functions are altered so that they satisfy the flux jump conditions, it seems to show  a reasonable order of convergence.

In this paper, we propose an improved version of the $P_1$ based IFEM
 by adding the line integral of flux  terms on each element. This technique resembles the discontinuous Galerkin (DG) method,
 however, our method  has much less degrees of freedom than
 the DG  methods since we use the same number of unknowns as the conventional $P_1$ finite element method.

 We prove $H^1$ and $L^2$  error estimates which are optimal both in order and regularity.
Numerical experiments were carried out for several examples, which show the robustness of our scheme.
\end{abstract}

\maketitle
\section{ Introduction }
In  recent years, there have been some developments of immersed finite element methods  for  elliptic  problems having an interface. These methods use meshes which do not necessarily align with the discontinuities of the coefficients \cite{Li2003}, \cite{Li2004}, thus
   violate a basic principle of triangulations in the conventional finite element methods \cite{Braess}, \cite{Ciarlet}.
However, when  the basis functions are modified so that they satisfy the interface conditions, they seem to work well
 \cite{Ch-Kw-We}, \cite{Li2003}, \cite{Li2004}.
 These methods were extended to the case of Crouzeix-Raviart $P_1$ nonconforming finite element method \cite{Crouzeix}
 by Kwak et al. \cite{Kwak-We-Chang}, and to
  the problems with nonzero jumps in \cite{Chang_Kwak}.
 Some related works on interface problems  can be found in \cite{Bramble}, \cite{Hou_L},  \cite{Hou_W_W}, \cite{Lai_Li_L}, \cite{LI:2006}, \cite{Oeverm_S_K},\cite{Yu}.

On the other hand, the discontinuous Galerkin methods (DG)  where one uses completely discontinuous basis functions
were developed and have been studied extensively, see \cite{Arnold-IP}, \cite{Ar-B-Co-Ma}, \cite{Da-Sun-Wh},
 \cite{Riv-Wh-Gi} and references therein. The DG methods work quite well for problems with discontinuous coefficient
 in the sense that they capture the sharp changes of the solutions well,
 yet they require large number of unknowns and  the meshes have to be aligned with the discontinuity.

 The purpose of this paper is to combine the advantages of the two methods.
We use a DG type idea of adding the consistency terms to the IFEM, thus proposing a modified version of IFEM based on the $P_1$ - Lagrange basis functions  on triangular grids.
 In spirit, it resembles \cite{A.P.Hansbo2002} in the sense that the standard linear basis functions are used
 for noninterface elements and  line integrals are added,
but in our method the line integrals along the edges, not along the interface, are added. Furthermore,
our method incorporate the flux jump conditions to the basis functions hence requires no extra unknowns along the interface as in \cite{A.P.Hansbo2002}.
%See also \cite{P.Hansbo_Lar2002}, where the DG stability term is combined with CR elt. for elasticity eq.

 We prove error estimates in the mesh dependent $H^1$ - norm and $L^2$ - norm
which are optimal both in the order and the regularity. We carry out various numerical tests to  confirm our theory and
compare the performance  with the unmodified scheme.

\section{ Preliminaries }
 Let $\Omega$ be a connected, convex polygonal domain in $\mathbb{R}^2$ which is divided into two subdomains
$\Omega^+$ and $\Omega^-$ by a  $C^2$ interface $\Gamma = \partial \Omega^+ \cap \partial \Omega^-$, see Figure
1.   % $=T,\Omega^+, \Omega^-, \Omega$
We assume that $\beta(x)$ is a positive function bounded below and above by two positive constants. Although our theory
applies to the case of nonconstant  $\beta(x)$,
we assume $\beta(x)$ is piecewise constant for the simplicity of presentation:
there are two positive constants $\beta^+, \beta^-$
such that $\beta(x)= \beta^+ $ on $\Omega^+$ and $\beta(x)=\beta^- $ on $\Omega^-$.
Consider the following elliptic interface problem
\begin{eqnarray}\label{TheprimalEq}
-\nabla\cdot (\beta(x)\nabla  u) &=& f ~~ \mathrm{   in}~  \Omega^s \quad (s=+,-)\\
u &=& 0  ~~ \mathrm{ on}~  \partial\Omega
\end{eqnarray}
with the jump conditions along the interface
\begin{eqnarray} \label{jump_condition}
[u]_\Gamma=0, ~~~ \left[\,\beta(x)\pd  un\,\right]_\Gamma=0,
\end{eqnarray}
where $f\in L^2(\Omega)$ and $u\in H^1_0 (\Omega)$ and the bracket $[\cdot]_\Gamma$ means the jump across the interface:
$$ [u]_\Gamma:= u|_{\Omega^+}-u|_{\Omega^-}. $$
\begin{figure}[t]
\begin{center}
      \psset{unit=2cm}
      \begin{pspicture}(-1,-1)(1,1)
        \pspolygon(0.9,0.9)(-0.9,0.9)(-0.9,-0.9)(0.9,-0.9)
        \psccurve(0.47,0) (0.2,0.22)(-0.6,-0.1)(-0.2,-0.5)(0.5,-0.11)
        \rput(0,0){\scriptsize$\Omega^-$}
        \rput(-0.3,0.5){\scriptsize$\Omega^+$}
        \rput(0.7,0){\scriptsize$\Gamma$}
      \end{pspicture}   \label{fig:doma0}
\caption{A domain $\Omega$ with interface}
\end{center}
\end{figure}
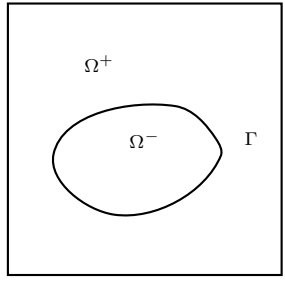

Let $p\geq1$ and  $m\geq 0$ be an integer.
For any domain $D$, we let  $W^m_p(D) $ be the usual  Sobolev space with (semi)-norms  and
 denoted by $|\cdot|_{m,p,D}$ and  $\|\cdot\|_{m,p,D}$.
 %When $p=2$, we write $H^m(D):=W^m_2(D)$ with the (semi)- norms $|\cdot|_{m,D}$ and
%$\|\cdot\|_{m,D}$.

 For $m\ge 1$, let
\begin{eqnarray*}
\wtilde{W}^m_p(D) &:=& \{\,u\in {W}^{m-1}_p(D)  : \,u|_{D\cap \Omega^s}\in
W^m_p(D\cap \Omega^s), s = +,-\,\}, %~~~p\geq1,\,\textrm{ $m\geq 0$ }
\end{eqnarray*}
with norms;
\begin{eqnarray*}
 |u|^p_{\wtilde{W}^m_p(D)} &:=& |u|^p_{m,p,T\cap \Omega^+} + |u|^p_{m,p,D\cap \Omega^-},\\
\|u\|^p_{\wtilde{W}^m_p(D)} &:=&\|u\|^p_{m-1,p,D}+|u|^p_{\wtilde{W}^m_p(D)}.
\end{eqnarray*}
%\begin{eqnarray*}
%\|u\|^2_{m,p,T} &:=& \|p\|^2_{m,p,T^-\cap\Omega^-} +\|p^2_{m,p,T^+\cap\Omega^+} +
%\|p^2_{m,p,T^-_r} +\|p^2_{m,p,T^+_r},\\
%|u|^2_{m,p,T} &:=& |p|^2_{m,p,T^-\cap\Omega^-} + |p|^2_{m,p,T^+\cap\Omega^+} +
%|p|^2_{m,p,T^-_r} + |p|^2_{m,p,T^+_r} .
%\end{eqnarray*}
%Now we  define $\wtilde{W}^m_p(\Omega)$ to be the space of all functions
%$u\in {W}^{m-1}_p(\Omega)$ such that $u|_{\Omega^s}\in \wtilde{W}^m_p(\Omega^s), s=+,-$
%equipped with the broken (semi)-norms
%$|u|_{\wtilde{W}^m_p(\Omega)}:=(\sum_{s=+,-}|u|^p_{\wtilde{W}^m_p(\Omega^s)})^{1/p}$ and   $\|u\|_{\wtilde{W}^m_p(\Omega)}:=(\sum_{s=+,-}\|u\|^p_{\wtilde{W}^m_p(\Omega^s)})^{1/p}$.
When $p=2$, we write $\wtilde{H}^m(D)$  and denote
the (semi)-norms by $|u|_{\wtilde{H}^m (D)}$ and $\|u\|_{\wtilde{H}^m (D)}$.
 $H^1_0(\Omega)$ is the subspace of $H^1(\Omega)$ with zero trace on the boundary.
Also, when some finite element
triangulation $\{\mathcal{T}_h\}$  is involved, the norms are understood as  piecewise norms $ (\sum_{T\in\mathcal{T}_h}|u|^p_{\wtilde{W}^m_p(T)})^{1/p}$ and  $(\sum_{T\in\mathcal{T}_h}\|u\|^p_{\wtilde{W}^m_p(T)})^{1/p}$, etc. If $p=2$, we denote them by  $|u|_{m,h}$ and $\|u\|_{m,h}$.
% When $p=2$, it is written as $\wtilde{H}^m(\Omega)$  and$\|u\|_{m,\Omega}$, resp.
 We also need some subspaces of $\wtilde{H}^2(T)$ and $\wtilde{H}^2(\Omega)$  satisfying the jump conditions:
\begin{eqnarray*}
\wtilde{H}^2_{\Gamma}(T)\hspace{-5pt} &:=& \hspace{-5pt}\{\,u\in H^1(T) : \, u|_{T\cap
\Omega^s}\in H^2(T\cap
\Omega^s),\,s = +,-,~\,\left[\beta \pd un\right]_\Gamma = 0 \text{ on } \Gamma\cap T\, \}\\
\wtilde{H}^2_{\Gamma}(\Omega)\hspace{-5pt} &:=& \hspace{-5pt} \{\,u\in H^1_0(\Omega) :
 \, u|_{T}\in\wtilde{H}^2_{\Gamma}(T),\,\, \forall T\in \mathcal{T}_h \}.
\end{eqnarray*}

 Throughout the paper, the constants $C, C_0, C_1$, etc., are generic constants  independent
of the mesh size $h$ and functions $u, v$ but may depend on the problem data  $\beta, f$ and $\Omega$, and are not necessarily the same
on each occurrence.

The usual  weak formulation for  the problem (\ref{TheprimalEq}) - (\ref{jump_condition}) is: Find $u\in
H^1_0(\Omega)$ such that
\begin{equation} \label{weakop}
\int_\Omega \beta(x) \nabla u\cdot \nabla v dx = \int_\Omega f v dx , ~~~ \forall v
\in H^1_0(\Omega).
\end{equation}

We have the following existence and regularity theorem for this problem; see \cite{Bramble}, \cite{Chen1998}, \cite{RS}.
\begin{theorem} \label{thm:reg}
Assume that $f\in L^2(\Omega)$. Then the variational problem (\ref{weakop}) has a unique
solution $u\in\wtilde{H}^2(\Omega)$ which satisfies
\begin{eqnarray}
\|u\|_{\wtilde{H}^2(\Omega)} \leq C \|f\|_{L^2(\Omega)}.
\end{eqnarray}
\end{theorem}
\section{ $P_1$-immersed finite element methods}
We briefly review  the immersed finite element space based on the $P_1$ - Lagrange basis functions (\cite{Li2003}, \cite{Li2004}).  Let $\{\mathcal{T}_h\}$ be the usual quasi-uniform
triangulations of the domain $\Omega$ by the triangles of maximum diameter $h$ which may not be aligned with the interface $\Gamma$.
 We call an element
$T\in\mathcal{T}_h$ an {\it interface element} if the interface $\Gamma$
passes through the interior of $T$, otherwise we call it a
{\it noninterface} element.  Let $\mathcal{T}^I_h$ be the collection of all interface elements. We assume that the interface meets
the edges of an interface element at no more than two points.

We  construct the local basis functions on each
element $T$ of the partition $\mathcal{T}_h$. For a noninterface
element $T\in\mathcal{T}_h$, we simply use the standard linear
shape functions on $T$ whose degrees of freedom are functional
values on the vertices of $T$, and use $\overline{S}_h(T)$ to
denote the linear spaces spanned by the three nodal basis
functions on $T$:
$$ \overline{S}_h(T) = \text{span}\{\,\phi_i\,:\,\phi_i \text{ is the standard
  linear shape function}\,\}$$
We let $\overline{S}_h(\Omega)$ denote the space of usual continuous, piecewise linear polynomials with vanishing
boundary values.

Now we consider a typical interface element $T\in \mathcal{T}_h^I$ whose geometric configuration
is given  as in Fig. \ref{fig:interel}. Here  the curve between the two points $D$ and $E$ is a part of the interface and
 $\overline{DE}$ is the  line segment  connecting the intersections of the interface and the edges.
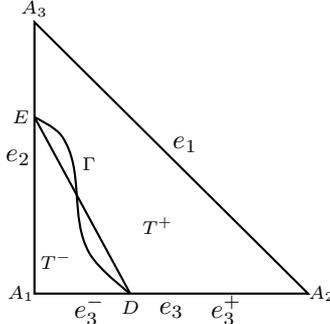
\begin{figure}[ht]
  \begin{center}
    \psset{unit=3.6cm}
    \begin{pspicture}(0,0)(1,1)
      \psset{linecolor=black} \pspolygon(0,0)(1,0)(0,1) \psline(0,0.65)(0.35,0)
      \pscurve(0,0.65)(0.1,0.58)(0.2,0.15)(0.35,0)
%      \psdot(0.5,0)
      \rput(0,1.05){\scriptsize$A_3$}
      \rput(-0.05,0){\scriptsize$A_1$}
      \rput(1.05,0){\scriptsize$A_2$}
      \rput(0.5,-0.06){$e_3$} \rput(0.2,-0.06){$e_3^-$} \rput(0.7,-0.06){$e_3^+$}
%      \psdot(0.5,0.5)
      \rput(0.55,0.55){$e_1$}
%      \psdot(0,0.5)
      \rput(-0.06,0.5){$e_2$}
      \rput(-0.05,0.65){\scriptsize$E$}
      \rput(0.08,0.12){\scriptsize$T^-$}
      \rput(0.45,0.25){\scriptsize$T^+$}
      \rput(0.35,-0.05){\scriptsize$D$}
      \rput(0.20,0.48){\scriptsize$\Gamma$}
%  \rput(-.35,0.65){\scriptsize$T_r^+$}
\pnode(-.3,0.6){a}
\pnode(0.12,0.5){b}
%\ncarc[linewidth=0.5\pslinewidth]{->}{a}{b}
%
%\pnode(-.3,.3){a}
%\pnode(0.22,0.21){b}
%\ncarc[linewidth=0.5\pslinewidth]{->}{a}{b}
%\rput(-.35,0.23){\scriptsize$T_r^-$}
    \end{pspicture}
    \caption{A typical interface triangle} \label{fig:interel}
\end{center}
\end{figure}

We construct  piecewise linear basis functions $\hat\phi_i, i=1,2,3$ of the form
\begin{eqnarray}
 &&\hat\phi_i(X) = \left\{%
\begin{array}{ll}
a^+ + b^+ x + c^+ y, & \textrm{ $X= (x,y)\in T^+$,}\\
a^- + b^- x + c^- y, & \textrm{ $X=(x,y)\in T^-$,} \\
\end{array}%
\right. \label{def:basis-1}
\end{eqnarray}  satisfying
\begin{eqnarray}
\, \hat\phi_i(A_j)& = &\delta_{ij},~~ j=1,2,3, \label{def:basis-2}\\
\,[\hat\phi_i(D)]& = &[\hat\phi_i (E)] =0, \\
\left[\beta \frac{\bd\hat\phi_i}{\bd \mathbf{n}} \right]_{\overline{\textrm{\tiny{$DE$}}}}&=&0. \label{def:basis-3}
\end{eqnarray}
These are  continuous, piecewise linear functions on $T$
satisfying  the flux jump condition along $\overline{DE}$, whose uniqueness and existence are known \cite{Ch-Kw-We}, \cite{Li2003}.
 \begin{remark} \label{rmk:derivative}
Since $\hat\phi_i$ is continuous, piecewise linear, it is clear that the tangential derivative  along  $\overline{DE}$  is continuous, i.e.,
\begin{equation*}
%\beta^+\pd{\hat\phi_i^+}{\bn_{\overline{DE}}} = \beta^-\pd{\hat\phi_i^-}{\bn_{\overline{DE}}}, ~ ~ ~ ~
\pd{\hat\phi_i^+}{\bt_{\overline{DE}}} = \pd{\hat\phi_i^-}{\bt_{\overline{DE}}},
\end{equation*}
where $\bt_{\overline{DE}}$ is the tangential vector to $\overline{DE}$.
\end{remark}
We denote by $\what{S}_h(T)$ the space of functions generated by $\hat{\phi}_i,\, i=1,2,3$  constructed  above.
Next  we define the global {\em immersed finite element space} $\what{S}_h(\Omega)$ to be the
set of all functions  $\phi\in L^2(\Omega)$ such that

$$
\left\{\begin{array}{l}
 \mbox{ }\phi\in \what{S}_h(T) \mbox{ if } T\in \mathcal{T}_h^I, \mbox{ and }  \phi\in \overline{S}_h(T) \mbox{ if } T \not\in \mathcal{T}_h^I, \\
\mbox{ having continuity at all vertices of the triangulation }\\
\mbox{ and vanishes on the boundary vertices.}
 \end{array}
\right\} $$
  We note that a function in $\what{S}_h(\Omega)$, in general,  is not
continuous across an edge common to two interface elements.
Let $H_h(\Omega):=H^1_0(\Omega)+\what{S}_h(\Omega)$
and equip it with the piecewise norms $|u|_{1,h}:=|u|_{\wtilde{H}^1(\Omega)},\,\|u\|_{1,h}:=\|u\|_{\wtilde{H}^1(\Omega)} $.
Next, we define the interpolation operator.
For any $u\in \wtilde{H}^2_{\Gamma}(T)$, we let $\hat I_h u\in \what{S}_h(T)$ be such that
$$\hat I_h u(A_i) = u(A_i),~~ i=1,2,3, $$ where $A_i,\,i=1,2,3$ are the vertices of
$T$ and we call $\hat I_h u$ the local \emph{interpolant} of $u$ in $\what{S}_h(T)$. We
naturally extend it to $\wtilde{H}^2_{\Gamma}(\Omega)$ by  $(\hat I_hu)|_{T} =\hat I_h (u|_T)$ for each $T$.
Then  we have the following approximation property \cite{Kwak-We-Chang}, \cite{Li2004}.%: For all  $u\in \wtilde{H}^2_{\Gamma}(\Omega),$

\begin{proposition}\label{approerror}
 There exists a constant $C>0$ such that %for all $u\in \wtilde{H}^2_{\Gamma}(T)$.
\begin{eqnarray}
\sum_{T\in \mathcal{T}_h} (\|u -\hat I_h u\|_{0,T} + h |u - \hat I_h u|_{1,T} )\leq C h^2 \|u\|_{\wtilde{H}^2(\Omega)}
\end{eqnarray} for all $u\in \wtilde{H}^2_{\Gamma}(\Omega)$.
\end{proposition}
%\correct{By carefully applying  scaling argument in the proof of the approximation property in \cite{Kwak-We-Chang},
% we can obtain the following  refined version of interpolation theorem.
%\begin{proposition}\label{approerror}
% There exists a constant $C>0$ such that %for all $u\in \wtilde{H}^2_{\Gamma}(T)$.
%\begin{eqnarray}
%  \|u -\hat I_h u\|_{0,T^s} + \sqrt{|T^s|} |u - \hat I_h u|_{1,T^s}  \leq C  |T^s|  \|u\|_{\wtilde{H}^2(T)}
%\end{eqnarray} for all $u\in \wtilde{H}^2_{\Gamma}(T)$.
%\end{proposition}
% \begin{proof}
%  Sketch of proof. Even though the degrees of freedom are distinct from each other (one is vertex d.of. the other is edge average d.o.f), the local behavior is exactly the same since they are both dealing with piecewise linear
%   functions. Thus it suffices to apply the scaling argument in Lemma 2.5 of \cite{Kwak-We-Chang}
% to the sets  $T^{s}, s= +,-$. The factor $h$ is now replaced by the square root of the
% area $T^s$. This completes the proof.%
% \end{proof}
%}
%Now we define the immersed finite element
 With $P_1$- Lagrange basis function introduced in \cite{Li2003}, \cite{Li2004}, the IFEM reads:
 \newline
{\bf ($P_1$-IFEM)} Find
$ {u}_h\in \what{S}_h(\Omega)$ such that
\begin{eqnarray} \label{dp}
a_h( {u}_h, {v}_h) = (f, {v}_h),~~~\forall\,
 {v}_h\in\what{S}_h(\Omega),
\end{eqnarray}
where
\begin{equation*}
a_h(u,v) =\sum_{T\in \mathcal{T}_h} \int_T \beta \nabla u \cdot \nabla  v\,
dx,~~~ \forall\,u,v\in H_h(\Omega).
\end{equation*}
The error estimate for this scheme is shown in \cite{Ch-Kw-We}.
%==============================================================================
%
\section{ Modified $P_1$-IFEM }

%This scheme has been tested  and seems to work well for certain examples \cite{Li2003}, \cite{Li2004}, and
%its convergence is analyzed in \cite{Ch-Kw-We}. However, later numerical tests show that this scheme is not robust
% for some other problems. The source of error seems to be the discontinuous line integrals along the edges which do not cancel.
In this section, we modify the  $P_1$-IFEM above by adding the line integrals for jumps of fluxes and functional values.
The method  resembles the discontinuous
 Galerkin methods (see \cite{Ar-B-Co-Ma}, \cite{Doug-Du}, \cite{Riv-Wh-Gi} and references therein) which use completely
 discontinuous basis functions,
 but the degrees of freedom in our method are much smaller than the DG methods since our method
 has the same number of basis functions as the conventional $P_1$-FEM.

In order to describe the new  method, we need some additional  notations.
Let the collection of all the edges of $T\in \mathcal T_h$ be denoted by $\mathcal{E}_h$ and
we split $\mathcal{E}_h$ into two disjoint sets; $\mathcal{E}_h = \mathcal{E}_h^o\cup \mathcal{E}_h^b, $
where $\mathcal{E}_h^o$ is the set of edges lying in the interior of $\Omega$,
and $\mathcal{E}_h^b$ is the set of edges on the boundary of $\Omega$.
In particular, we denote the set of edges cut by the interface $\Gamma$ by $\mathcal{E}_h^I$.
For every  $e\in\mathcal{E}_h^o$, there are two element $T_1$ and $T_2$ sharing $e$ as a common edge.
 Let $\bn_{T_i},  i=1,2$ be the
unit outward normal vector to the boundary of $T_i$, but  for the edge $e$,
we choose a direction of the normal vector, say $\bn_e=\bn_{T_1}$  and fix it once and for all.
For functions $v$ defined on $ T_1 \cup T_2 $, we let $[\cdot]_e$ and $\{\cdot\}_e$  denote the jump and average across
$e$ respectively,  i.e.
$$[v]_e = v^1 - v^2, \{v\}_e = \frac{1}{2} (v^1 + v^2). $$
We also need  the  mesh dependent norm $\tnorm{\cdot}$ on the space $H_h(\Omega)$,
\begin{eqnarray*}\label{energyN}
 \tnorm {v}^2 &:=&\sum_{T\in \mathcal T_h} \left(\left\|\sqrt\beta v\right\|_{0,T}^2+
\left\|\sqrt\beta\nabla v\right\|_{0,T}^2 \right) \\
&&+ \sum_{e\in\mathcal{E}_h^o} \left(h\left\|\{\sqrt\beta \nabla v\ \cdot \bn_e\}_e\right\|_{0,e}^2
 +  h^{-1}\left\|[\sqrt\beta v]_e\right\|_{0,e}^2 \right) .
\end{eqnarray*}

Multiplying both sides of the equation (\ref{TheprimalEq}) by $v\in H^1(T)$, applying Green's formula
and adding, we get
 \begin{eqnarray*}
 \sum_{T\in \mathcal T_h} \left( \int_{T} \beta \nabla u\cdot \nabla v dx
 - \int_{\partial T}  \beta \nabla u \cdot \bn_T \, v ds \right)
 =\int_{\Omega} fv dx.
 \end{eqnarray*}
% $$a_1b_1 - a_2b_2 = \frac12 [(a_1 + a_2)(b_1- b_2) + (a_1 - a_2)(b_1 + b_2)]$$
 By using the preassigned normal vectors $\bn_e$ and adding the
unharmful term  $ \eps\int_{e} \{\beta \nabla v \cdot \bn_e \}_{e} [u]_{e}$ for any $\eps$, we see the above equation becomes
 \begin{multline}\label{DG_conti}
 \sum_{T\in \mathcal T_h}  \int_{\T} \beta \nabla u\cdot \nabla v dx - \sum_{e\in \mathcal{E}_h^o} \int_{e}
 \{\beta \nabla u \cdot \bn_e \}_{e} [v]_{e}ds\\
 + \eps\sum_{e\in \mathcal{E}_h^o} \int_{e}\{\beta \nabla v\cdot\bn_e\}_{e}[u]_{e}ds
 =  \int_\Omega f v dx
 \end{multline} which is valid for $v\in L^2(\Omega)$ such that $v\in H^1(T)$ for all $T\in \mathcal T_h$.
We define the following bilinear forms
\begin{eqnarray*} \label{a-e-form}
% a(u,v) &:=& \sum_{T\in \mathcal T_h} \int_T\beta\nabla u \cdot \nabla v \,dx  \\
  b_\eps(u,v) &:=& -\sum_{e\in \mathcal{E}_h^o} \int_{e}\left\{\beta\nabla u\cdot\bn_e\right\}_{e}[v]_{e} ds
 +\eps\sum_{e\in \mathcal{E}_h^o} \int_e\left\{\beta\nabla v\cdot\bn_e\right\}_{e}[u]_{e} ds,  \\
  j_\sigma(u,v) &:=&\sum_{e\in \mathcal{E}_h^o} \int_{e}\frac \sigma h [u]_{e}[v]_{e} ds, \mbox { for some } \sigma>0\\
%j_\sigma(u,v) &:=&\sum_{e\in \mathcal{E}_h^o} \sum_{s=+,-} \int_{e^s}\frac \sigma {|\tilde{e}^s|} [u]_{e}[v]_{e} ds , \mbox { for some } \sigma>0\\
   a_\eps (u,v) &:= & a_h(u,v) + b_\eps(u,v) +j_\sigma(u,v).
\end{eqnarray*}
Now, for each $\eps=0$, $\eps=-1$ and $\eps=1$, we define the modified $P_1$-IFEM for the problem
 (\ref{TheprimalEq})-(\ref{jump_condition}): \\
{\bf(Modified $P_1$-IFEM)}
Find $u_h^m \in \what{S}_h(\Omega)$ such that
\begin{equation}\label{Modi-IFEM}
a_\eps(u_h^m , v_h) = (f,v_h), ~~~ \forall v_h \in \what{S}_h(\Omega).
 \end{equation}
This is similar to a class of DG methods, corresponding to  IP, SIPG,
 NIPG and OBB (\cite{Arnold-IP}, \cite{Doug-Du}, \cite{Da-Sun-Wh}, \cite{OBB}), if $\eps=0$, $\eps=-1$, $\eps=1$,  and $\eps=0,\sigma=0$, respectively.

% Mine is OBB-DG
\begin{remark}
For the line integrals in $b_\eps(u,v)$, it suffices to consider the integrals on the edges of the interface elements only since  both $[u_h], [v_h]$ vanish for $e\in \mathcal{E}_h^o\setminus \mathcal{E}_h^I$.
\end{remark}

\section{ Error analysis } % of modified $P_1$ IFEM
In this section, we prove an optimal order of error estimates in $H^1$ and $L^2$-norms of our schemes.
For simplicity, we present the case with  $\eps=-1$ only. All other cases are similar.
Also, we assume the smooth interface is replaced by piecewise line segment on each element.

We need the well-known inverse inequality and trace theorem (\cite{Arnold-IP}, \cite{Ciarlet}):
\begin{lemma}
There exist positive constants $C_0,C_1$ independent of the function $v_h$ such that
for all $v_h\in P_k(T)\cup \hat S_h(T)$,
\begin{align} \label{trace_ineq}
\|v_h\|_{1,T}^2\le C_0 h^{-2} \|v_h\|_{0,T}^2,&\quad \|v_h\|_{0,\partial T}^2 \le C_1 h^{-1}\|v_h\|_{0,T}^2.
% &\|v\|_\infty \le C h^{-1} \|v\|_{0,T}
\end{align}
There exists a positive constant $C_2$ independent of the function $v$ such that
for all $ v\in H^1(T)$ %Trace inequalities.
\begin{align} \label{trace-edge}
\|v\|_{0,e}^2\le C_2(h^{-1} \|v\|_{0,T}^2+ h|v|_{1,T}^2).
\end{align}
\end{lemma}
Now we show the following interpolation error estimate for the mesh dependent norm $\tnorm{\cdot}$.
\begin{proposition} \label{approx-inter} There exist positive constants $C$, $C_I$ independent of the function $u$ such that
for all $ u\in  H^1(\Omega)\cap  \wtilde{H}^2_{\Gamma}(\Omega)$,
\begin{equation}\label{inter-edge}\sum_{e\in\mathcal{E}_h^o} h\left\|\{\nabla (u-\hat I_h u) \cdot \bn_e\}_e\right\|_{0,e}^2 +\sum_{e\in \mathcal{E}_h^o} h^{-1}\left\|[u- \hat I_h u]_e\right\|_{0,e}^2 \le
Ch^{2}\|u\|_{\widetilde{H}^2(\Omega)}^2.\end{equation}
Consequently, we have
\begin{equation}\label{inter-trip}
  \tnorm{u-\hat I_h u} \le C_I h \|u\|_{\widetilde{H}^2(\Omega)}.
\end{equation}
\end{proposition}

\begin{proof}
  We first consider $\nabla (u-\hat I_h u)$.
Since $\nabla (u-\hat I_h u)$ is not in $H^1(T)$, we cannot apply  (\ref{trace-edge})  directly.
Instead, we decompose it as
$$ \nabla (u-\hat I_h u)= (\nabla (u-\hat I_h u)\cdot\bn_\Gamma)\bn_\Gamma+ (\nabla (u-\hat I_h u)\cdot\bt_\Gamma) \bt_\Gamma :=\bw+\bz,   $$
where $\bn_\Gamma$ and $\bt_\Gamma$ are the unit normal and tangent vector to the interface $\Gamma$, respectively.
We have
\begin{eqnarray*}
 \left\|\nabla (u-\hat I_h u)\cdot\bn_e\right\|_{0,e}^2
 &\le & \| \bw \cdot\bn_e\|_{0,e}^2+ \| \bz\cdot\bn_e\|_{0,e}^2 \\
 &\le & \frac{1}{\beta_{min}^2}\left\|\beta \bw \cdot\bn_e\right\|_{0,e}^2+ \| \bz\cdot\bn_e\|_{0,e}^2.
\end{eqnarray*}
We can easily check that $\beta \bw $ is in $H^1(T)$.
 For the smoothness of $\bz $ we proceed as follows: Since $u\in H^1(T)$, we have $ \left. (\nabla u \cdot \bt_\Gamma) \right|_{T^+\cap \Gamma} = \left. (\nabla u \cdot \bt_\Gamma) \right|_{T^-\cap \Gamma}$.
 Hence $ \left. (\nabla u \cdot \bt_\Gamma) \right|_T$   has well defined trace on $\Gamma$, which implies $\nabla u\cdot\bt_\Gamma$  is in $\tilde H^1(T)$ also.
Therefore, we can apply (\ref{trace-edge}) to $\beta \bw\cdot\bn_e$ and $\bz \cdot\bn_e$. Hence
\begin{eqnarray*}
&&h\left\|\nabla (u-\hat I_h u)\cdot\bn_e\right\|_{0,e}^2 \\
&\le& \frac{C_2}{\beta_{min}^2} \left(\|\beta \bw \cdot\bn_e\|_{0,T}^2 + h^2 |\beta \bw \cdot\bn_e|_{1,T}^2 \right)+ C_2 \left(\| \bz\cdot\bn_e\|_{0,T}^2  + h^2|\bz\cdot\bn_e|_{1,T}^2\right) \\
&\le& C_2 \left(\frac{\beta_{max}^2}{\beta_{min}^2}\left(\|\bw \|_{0,T}^2+ h^2 |\bw |_{\tilde H^1(T)}^2  \right)
+  \|\bz\|_{0,T}^2 + h^2 |\bz|_{\tilde H^1(T)}^2 \right) \\
&\le& C_2 \ \alpha^2 \left( \left\| \nabla(u-\hat I_h u) \right\|_{0,T}^2 + h^2 \left| \nabla(u-\hat I_h u)\right|_{\tilde H^1(T)}^2\right)\\
&\le& C h^2  \left\| u \right\|_{\tilde H^2(T)}^2 ,
\end{eqnarray*}
where $\beta_{min} = \min(\beta^+, \beta^-)$ , $\beta_{max} = \max(\beta^+, \beta^-)$ and we have set $\alpha = \frac{\beta_{max}}{\beta_{min}}$.
Here   Proposition \ref{approerror} was used to derive  the last estimate.  

The estimate of the second term follows easily from  (\ref{trace-edge})
  and Proposition \ref{approerror}:
\begin{eqnarray*}
h^{-1}\left\| u-\hat I_h u \right\|_{0,\partial T}^2 &\le& C_2 \left(h^{-2} \left\| u-\hat I_h u\right\|_{0,T}^2 +  \left| u-\hat I_h u\right|_{\tilde H^1(T)} ^2 \right) \\
& \le & C h^2 \| u\|_{\tilde H^2(T)} ^2.
\end{eqnarray*}
Thus, the estimate (\ref{inter-trip}) follows.
\end{proof}
  The following discrete Poincar\'e inequality holds for $\hat S_h(\Omega)$, (see \cite{Ch-Kw-We}).
\begin{lemma} \label{poin-sh}
There exists a constant $C_p>0$ such that
\begin{equation}
 C_p\|v_h\|_{0,\Omega}^2 \le |v_h|_{1,h}^2,\  ~~~ \forall v_h \in \hat S_h(\Omega).
\end{equation}
\end{lemma}
%{\it (Boundedness) }
%Extend the def of interpol.
Now we show some basic properties of $a_\eps (\cdot,\cdot)$. Clearly, $ a_\eps(\cdot,\cdot)$ is bounded on $H_h(\Omega)$ with respect to  $\tnorm{\cdot}$:
$$ |a_\eps (u,v)|\le C_b \tnorm{u}\tnorm{v} ,\quad  ~~~ \forall u,v\in H_h(\Omega).$$

Next, we prove the  coercivity of the form $a_\eps(\cdot,\cdot)$ on the space $\hat S_h(\Omega)$.
We need a lemma.
\begin{lemma} \label{lem:deriv}
For all  $v\in \hat S_h(\Omega)$, there exists a positive constant C independent of h such that
\begin{equation}
\sum_{e\in \mathcal{E}_h^o}h\left\|\left\{\sqrt{\beta}\nabla v\cdot\bn\right\} \right\|^2_{0,e} \le C_1 \alpha \sum_{T\in \mathcal T_h} \left\|\sqrt{\beta}\nabla v\right\|_{0,T}^2
\end{equation} where $\alpha=\frac{\beta_{max}}{\beta_{min}}$ is the same as before.
\end{lemma}
\begin{proof} {
%Let $\bn_\Gamma, \bt_\Gamma$ be the unit normal, tangent vector respectively to the interface.
We decompose $\nabla v_h$ as
\begin{equation}
 \nabla v_h= (\nabla v_h\cdot\bn_\Gamma)\bn_\Gamma+ (\nabla v_h\cdot\bt_\Gamma) \bt_\Gamma :=\bw+\bz.
 \end{equation} The rest of the proof is almost the same as that of (\ref{inter-edge}).}
\end{proof}
\begin{proposition} \label{lem:Coer} There exists a positive constant $C_c$ independent of $v_h$ such that
for all  $v_h\in \hat S_h(\Omega)$ the following holds:
$$ a_\eps (v_h,v_h) \ge C_c \tnorm{v_h}^2. $$
\end{proposition}
\begin{proof}
First of all, we consider $b_\eps(v_h,v_h)$, the second part of $a_\eps (v_h,v_h)$.
%$$ \sum_{e\in \mathcal{E}_h^o}\sum_{s=+,-} |\tilde{e}^s| \ \|\left\{\sqrt{\beta}\nabla v_h\cdot\bn\right\} \|^2_{0,e^s}\le
%C_1 \sum_{T\in \mathcal T_h} \|\sqrt{\beta}\nabla v_h\|_{0,T}^2 $$ holds.
%Let $\beta_1=\max(\beta^+,\beta^-)$ and  $\beta_0=\min(\beta^+,\beta^-)$.
By
Lemma \ref{lem:deriv}, Cauchy-Schwarz and arithmetic-geometric inequality, we have
\begin{eqnarray*}\nonumber
 &&\sum_{e\in \mathcal{E}_h^o} \int_{e}\left\{\beta\nabla v_h\cdot\bn\right\}[v_h] ds\\
  &\le& \left( \sum_{e\in \mathcal{E}_h^o} h \left\|\left\{\beta\nabla v_h \cdot\bn\right\}\right\|_{0,e}^2  \right)^{1/2}
  \left(\sum_{e\in\mathcal{E}_h^o} h^{-1}\left\|[v_h]\right\|_{0,e}^2\right)^{1/2}\\
&\le&  \left(C_1\alpha\sum_{T\in \mathcal T_h} \left\|\sqrt{\beta}\nabla v_h\right\|_{0,T}^2\right)^{1/2}\left(\sum_{e\in\mathcal{E}_h^o} h^{-1}\left\|[\sqrt{\beta}v_h]\right\|_{0,e}^2\right)^{1/2}\nonumber \\
&\le&  \frac{\gamma}{2} \left(\sum_{T\in \mathcal T_h} \left\|\sqrt{\beta}\nabla v_h\right\|_{0,T}^2\right) + \frac{C_1\alpha}{2\gamma}
 \left(\sum_{e\in\mathcal{E}_h^o} h^{-1}\left\|[\sqrt{\beta}v_h]\right\|_{0,e}^2\right)% \label{C-beta}
\end{eqnarray*}
for every $\gamma>0$.
 Hence by Lemma \ref{poin-sh}, we have
\begin{eqnarray*}
 \hspace{-15pt} a_\epsilon(v_h,v_h)
&=& a_h(v_h,v_h) + b_\eps(v_h,v_h) +j_\sigma(v_h,v_h) \\
&=& \sum_{T\in \mathcal T_h} \int_T\beta\nabla v_h \cdot \nabla v_h\,dx - 2\sum_{e\in \mathcal{E}_h^o} \int_{e}\left\{\beta\nabla v_h\cdot\bn\right\}[v_h] ds +\sum_{e\in \mathcal{E}_h^o} \int_{e}\frac \sigma{h} [v_h]^2ds  \\
%&\ge& \sum_{T\in \mathcal T_h} \beta(1-\gamma) |v|_{1,T} + \sum_{e\in \mathcal{E}_h^o} \left(\sigma-\frac{C_1\beta}{\gamma}\right)\frac{1}{h}\|[v]\|_{0,e}^2 \\
&\ge& \frac{C_p}2 \left\|\sqrt{\beta}v_h\right\|_{0,\Omega}^2+(\frac{1}{2}-\gamma ) \left|\sqrt{\beta} v_h\right|_{1,h}^2 + \left(\sigma_0-\frac{C_1\alpha}\gamma \right)\sum_{e\in \mathcal{E}_h^o} \frac{1}{h} \left\|[\sqrt{\beta} v_h]\right\|_{0,e}^2\\
&\ge& \frac{C_p}2 \left\|\sqrt{\beta}v_h\right\|_{0,\Omega}^2+(\frac{1}4-\gamma) \left|\sqrt{\beta} v_h\right|_{1,h}^2+ \frac{1}{4C_1\alpha} \sum_{e\in \mathcal{E}_h^o} h \left\|\left\{\sqrt{\beta}\nabla v_h\cdot\bn\right\} \right\|^2_{0,e}\\
&&+\left(\sigma_0-\frac{C_1\alpha}\gamma \right)\sum_{e\in \mathcal{E}_h^o} \frac{1}{h} \left\|[\sqrt{\beta}v_h]\right\|_{0,e}^2,
\end{eqnarray*}
where we have set $\sigma_0=\sigma/\beta$.
If we choose $\gamma=\frac{1}{8}$ and $\sigma_0$ large enough so that $\left(\sigma_0- 8 C_1 \alpha \right)\ge
 \frac{1}8. $
Then with  $C_c:=\min \left(\frac{C_p}2,\, \frac{1}8, \, \frac{1}{4C_1\alpha} \right) $,  we have
$$ a_\eps (v,v) \ge C_c \tnorm{v}^2.$$
\end{proof}
%
%\begin{eqnarray*}
% \hspace{-15pt} a_\epsilon(v,v)
%&=& a_h(v,v) + b_\eps(v,v) +j_\sigma(v,v) \\
%&=& \sum_{T\in \mathcal T_h} \int_T\beta\nabla v \cdot \nabla v\,dx - 2\sum_{e\in \mathcal{E}_h^o} \int_{e}\left\{\beta\nabla v\cdot\bn\right\}[v] ds +\sum_{e\in \mathcal{E}_h^o} \int_{e}\frac \sigma h [v]^2ds  \\
%&\ge& \sum_{T\in \mathcal T_h} \beta(1-\gamma) |v|_{1,T} + \sum_{e\in \mathcal{E}_h^o} \left(\sigma-\frac{C_1\beta}{\gamma}\right)\frac{1}{h}\|[v]\|_{0,e}^2 \\
%&\ge& \frac{\beta_0 C_p}2 \|v\|_{0,\Omega}^2+\beta(\frac{1}{2}-\gamma ) |v|_{1,h}^2 + \left(\sigma-\frac{C_1\beta_1}\gamma \right)\sum_{e\in \mathcal{E}_h^o} \frac{1}{h} \|[v]\|_{0,e}^2.
%\end{eqnarray*}
%Hence, by (\ref{trace_ineq}),
%\begin{eqnarray*}
%a_\epsilon(v,v)
%&\ge& \frac{\beta_0 C_p}2 \|v\|_{0,\Omega}^2+(\frac{\beta_0}4-\gamma ) |v|_{1,h}^2+ \frac{\beta_0}{4C_1} \sum_{e\in \mathcal{E}_h^o}h\|\left\{\nabla v\cdot\bn\right\} \|^2_{0,e}\\
%&&+\left(\sigma-\frac{C_1\beta_1}\gamma \right)\sum_{e\in \mathcal{E}_h^o} \frac{1}{h} \|[v]\|_{0,e}^2.
%\end{eqnarray*}
%If we choose $\gamma=\frac{\beta_0}{8 }$ and $\sigma$ large enough so that $\left(\sigma-\frac{C_1\beta_1}\gamma \right)\ge
% \frac{\beta_0}8. $
%Then with  $C_c:=\min \left(\frac{\beta_0 C_p}2,\, \frac{\beta_0}8, \, \frac{\beta_0}{4C_1}\right) $,  we have
%$$ a_\eps (v,v) \ge C_c \tnorm{v}^2.$$
%\end{proof}
  {
\begin{remark}
We can take any positive $\sigma$ when $\eps=1$, because $b_\eps(v,v)$ becomes zero.
If $\eps=0$ or $-1$, it seems that $\sigma>0$ must be large enough to show the coercivity.
However, small positive $\sigma$ or even $\sigma=0$ works for all the cases we have tested. This is in contrast to the usual DG schemes, where sufficiently large $\sigma$ is necessary.
The reason seems to be that, unlike the usual DG, the term $b_\eps(v,v)$ is small enough to be dominated by $a_h(v,v)$, since  the jump $[v]$  vanishes at the vertices of each $T\in \mathcal T_h$.
In fact, using the techniques in
\cite{CMAME-unp} and  the proof of Proposition  \ref{approx-inter}
 we can show $|b_\eps(v,v)|\le C (h|\log h|)^{1/2} \|v\|^2_{1,h}$, but the details are complicated. This will be shown in the subsequent paper.
\end{remark}}

\subsection{ $H^1$-error analysis} $\newline$
 First we check that the modified $P_1$-IFEM is consistent.
\begin{lemma}\label{lem:consistency} Let $u$ be the solution of (\ref{TheprimalEq})-(\ref{jump_condition}) and let
$u_h^m$ be the solution of (\ref{Modi-IFEM}). For any $ v_h \in \what{S}_h(\Omega)$, we have
\begin{equation}\label{Consisten-eps}
a_\eps (u,v_h)=(f,v_h).
\end{equation}
In other words,
$$ a_\eps (u-u_h^m,v_h) = 0. $$
\begin{proof}
By (\ref{DG_conti}), the definition of the $a_\eps $ form and the homogeneous jump condition of $u$, we have
\begin{equation*}
 a_\eps (u,v_h) - a_\eps (u_h^m,v_h) = a_\eps (u,v_h)-(f,v_h) = \sum_{e\in \mathcal{E}_h^o} \int_{e}\frac \sigma h [u]_{e}[v_h]_{e} ds = 0 .
\end{equation*}
\end{proof}
\end{lemma}
Now we can prove the $H^1$-error estimate which is  optimal both in order and the regularity.
\begin{theorem}\label{Thm:H1} Let $u$ be the solution of (\ref{TheprimalEq})-(\ref{jump_condition}) and let
$u_h^m$ be the solution of (\ref{Modi-IFEM}). Then there exists a positive constant $C$ independent of $u$ and $h$ such that
 \begin{eqnarray*}
 \tnorm{u - u_h^m} &\le & C h\|u\|_{\widetilde{H}^2(\Omega)}.  \label{H1-error}
\end{eqnarray*}
\end{theorem}
\begin{proof}
By Proposition \ref{lem:Coer}, (\ref{Consisten-eps}) and boundedness of  $a_\eps (\cdot,\cdot)$ with respect to $\tnorm{\cdot}$, we have
\begin{eqnarray*}
\tnorm{u_h^m-\hat I_hu}^2 &\le &  C_c^{-1} a_\eps (u_h^m-\hat I_hu,u_h^m-\hat I_hu)\\
&= & C_c^{-1} a_\eps (u-\hat I_hu,u_h^m-\hat I_hu)\\
 &\le & C_c^{-1} C_b \tnorm{u-\hat I_h u} \tnorm {u_h^m-\hat I_h u}.
  \end{eqnarray*}
By the triangle inequality and Proposition \ref{approx-inter}, we get
\begin{eqnarray*}
\tnorm{u-u_h^m} &\le& \tnorm{u-\hat I_h u}+ \tnorm{u_h^m-\hat I_hu} \\
&\le& (C_c^{-1} C_b +1)C_I h \|u\|_{\widetilde{H}^2(\Omega)} .
\end{eqnarray*}
\end{proof}
\subsection{ $L^2$-error analysis}
%Now we have the following.
\begin{theorem}
For the solution $u_h^m$ of (\ref{Modi-IFEM}), there exists a positive constant $C$ independent of $u$ and $h$ such that
 \begin{eqnarray*}
 \|{u - u_h^m}\|_{L^2(\Omega)} &\le & C h^2\|u\|_{\widetilde{H}^2(\Omega)}.  \label{L2-error}
\end{eqnarray*}
\end{theorem}
\begin{proof}
Consider the dual equation:
\begin{eqnarray*}
-\nabla (\beta\nabla \Psi) &=& w  ~~\quad \mathrm{  in}~  \Omega^s \,\, (s=+,-) \\
\left[\Psi\right]_\Gamma &=&0, \\
\left[\,\beta(x)\pd \Psi n\,\right]_\Gamma&=&0, ~~  \\
\Psi &=& 0  ~~\quad \mathrm{ on}~  \partial\Omega.
\end{eqnarray*}
Then by Theorem \ref{thm:reg} the solution satisfies
\begin{equation}\label{rg-dual}
\|\Psi\|_{\tilde H^2(\Omega)} \le C \|w\|_{L^2(\Omega)}.
\end{equation}

Let $\Psi_h$ be the modified IFEM solution of this problem. Then with $e_h:=u- u_h^m$,
we have by Lemma \ref{lem:consistency}
 $$   (e_h,w)=a_\eps(e_h,\Psi)  = a_\eps(e_h,\Psi-\Psi_h). $$  Then by boundedness of $a_\eps$, Theorem \ref{Thm:H1} and (\ref{rg-dual})
\begin{equation*}
 |(e_h,w) | \le C \tnorm{e_h}\tnorm{\Psi-\Psi_h}\le C h\|\Psi\|_{\tilde H^2(\Omega)} \tnorm{e_h} \le Ch \|w\|_{L^2(\Omega)} \tnorm{e_h}.
\end{equation*}
 Taking $w=e_h$, we obtain
 $$\|u - u^m_h \|_{L^2(\Omega)} \le Ch \tnorm{u-u^m_h} \le C h^2 \|u\|_{\widetilde{H}^2(\Omega)}. $$
\end{proof}
\section{Numerical Experiments}
For  numerical tests, we solve the problem
(\ref{TheprimalEq})-(\ref{jump_condition}) on the rectangular domain $\Omega=[-1,1]\times[-1,1]$ partitioned into
unform right triangles with $h_x = h_y = 1/2^{n-1}$ for $n=4, \cdots ,10$.
Three types of interface problems are considered with various values of parameter $\beta$.
We measured $\|u-u_h\|_{0}$ and $\|u-u_h\|_{1,h}$ which are very close to the theoretical
orders of convergence, $2$ and $1$ respectively.   Although not reported, we also measured $\sum_e\|u-u_h\|_{0,e}$ and $\sum_e\|\partial (u-u_h)/\partial n\|_{0,e}$, the orders of which agree with the theoretical value $1.5$ and $0.5$ respectively.
Moreover, we observe the second order convergence in $L^\infty$ norm also.

 %with our new scheme but slight deterioration is shown for some problems with the original $P_1$-IFEM.
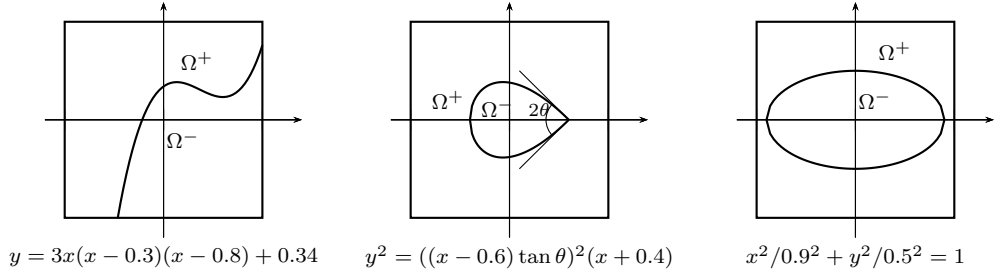
\begin{figure}[hbt]%%%% h=here, or ht, t, bottom
\begin{center}
\begin{pspicture}(-0,-1.8)(10,1.6)\footnotesize
%%%%
\psset{xunit=1.3, yunit=1.3}
%%%%%%%
\rput(0.5,0){ \psaxes[ticks=x, labels=none,linewidth=.5\pslinewidth]{->}(0,0)(-1.2,-1.2)(1.4,1.2)
\pspolygon(-1,-1)(1,-1)(1,1)(-1,1)
\psplot{-0.464}{1}{x 0.3 sub x mul 3 mul x 0.8 sub mul 0.34 add} %[fillstyle=solid,plotpoints=200]
 \rput(0.2,-0.2){$\Omega^-$}  \rput(0.34,0.6){$\Omega^+$}
\rput(0,-1.4){$y=3x(x-{0.3})(x-0.8)+0.34$}
}

\rput(4.0,0){
\psaxes[ticks=x, labels=none,linewidth=.5\pslinewidth]{->}(0,0)(-1.2,-1.2)(1.4,1.2)
\pspolygon(1,-1)(1,1)(-1,1)(-1,-1)(1,-1)

\rput(-0.62,0.2){$\Omega^+$} \rput(-0.12,0.12){$\Omega^-$}
\parametricplot{-.4}{.6}{t t  0.6 sub 2 exp  t 0.4 add mul   sqrt } % y=\sqrt{(x-1)^2x}
\parametricplot{-.4}{.6}{t t  0.6 sub 2 exp  t 0.4 add mul   sqrt -1 mul }
\rput(0.6,0){
\psline[linewidth=.5\pslinewidth](-0.5,0.5)(0,0)(-0.5,-0.5)
\psarc[linewidth=.5\pslinewidth]{-}(0,0){0.3} {135}{225}
\rput(-0.3,0.1){\scriptsize$2\theta$}
}
\rput(0.1,-1.4){$y^2 =((x-0.6)\tan\theta )^2(x+0.4)$}
 }

\rput(7.5,0){\psaxes[ticks=x, labels=none,linewidth=.5\pslinewidth]{->}(0,0)(-1.2,-1.2)(1.4,1.2)
\pspolygon(-1,-1)(1,-1)(1,1)(-1,1)
\psplot{-0.9}{0.9}{ 1 x .9 div 2 exp sub sqrt .5 mul} %[fillstyle=solid,plotpoints=200]
\psplot{-0.9}{0.9}{ 1 x .9 div 2 exp sub sqrt -.5 mul}
 \rput(0.2,0.2){$\Omega^-$}  \rput(0.4,0.7){$\Omega^+$}
\rput(0.,-1.4){$x^2/0.9^2+y^2/0.5^2=1$}
}
\end{pspicture}
\caption{Interfaces of Examples 1,2 and 3}\label{figure1.24}

\end{center}
\end{figure}
\begin{example}  [Cubic curve] \label{example:cubic}
  {The 0-set of function $L(x,y)=y-3x(x-0.3)(x-0.8)-0.34$ is used in this example as the interface.}
%We take the level set of $L(x,y)=y-3x(x-0.3)(x-0.8)-0.34$ as an  interface.
The exact solution is
$u=L(x,y)/\beta$, where $\beta=\beta^\pm$ on  $\Omega^\pm$. We test the cases when $ \beta^+/\beta^-=10$ and $1000$.
\end{example}
%\begin{figure}[htb]
%\centerline{ \includegraphics[width=8cm]{cubic_curve_level6}}
%\caption{Solution $u_h$ for Example \ref{example:cubic} ($\beta^- = 1$, $\beta^+ = 10$, $   {1/h_x = 32}$)}
%\end{figure}
%\begin{figure}
%    \centering
%    \begin{subfigure}[b]{0.45\textwidth}
%        \includegraphics[width=6.0cm]{Cubic_original_}
%        \caption{$P_1$-IFEM}
%    \end{subfigure}
%    \begin{subfigure}[b]{0.45\textwidth}
%        \includegraphics[width=6.0cm]{Cubic_modified}
%        \caption{Modified $P_1$-IFEM}
%    \end{subfigure}
%    \begin{adjustwidth}{1.0cm}{}
%    \caption[c]
%      {\begin{minipage}[t]{.8\linewidth}Error Surface for Example \ref{example:cubic} (Cubic Curve) \\( {$\beta^-=1$, $\beta^+=10$, $1/h_x = 128$}) \end{minipage}}\label{fig_cubic_surf}
%      %{\small Error Surface for Example \ref{example:cubic}(  {$\beta^-=1$, $\beta^+=10$, $1/h_x = 128$})}
%      \end{adjustwidth}
%
%\end{figure}

\begin{table}
\small
\begin{tabular}{|c|c||c|c||c|c||c|c|}
\hline   & $1/h_x$& \small{$\|u-u_h\|_{0}$} & order & \small{$\|u-u_h\|_{1,h}$} & order &\small{$\|u-u_h\|_{\infty}$}& order\\
\cline{2-8}
          &   8  & 1.344e-2 &       & 3.315e-1 &       &2.761e-2&	      	\\ %\cline{2-8}
          &  16  & 3.453e-3 & 1.961 & 1.709e-1 & 0.955 &8.715e-3&	1.663 \\ %\cline{2-8}
\multirow{2}{*}{$P_1$-IFEM} &  32   & 8.9002-4 & 1.956 &8.727e-3&   0.970 &3.069e-3&	1.506 \\ %\cline{2-8}
          &  64  & 2.161e-4 & 2.043 & 4.507e-2 & 0.953 &1.295e-3&	1.245  \\ %\cline{2-8}
          & 128  & 5.541e-5 & 1.963 & 2.347e-2 & 0.941 &5.786e-4&	1.162  \\ %\cline{2-8}
          & 256  & 1.851e-5 & 1.582 & 1.288e-2 & 0.865 &3.598e-4&	0.686  \\ %\cline{2-8}
          & 512  & 8.193e-6 & 1.176 & 7.297e-3 & 0.820 &1.776e-4&	1.018  \\ %\cline{2-8}
\hline  \hline & $1/h_x$& \small{$\hspace{2pt}\|u-u_h^m\|_{0}\hspace{1pt}$} & order & \small{$\hspace{1pt}\|u-u_h^m\|_{1,h}$} & order & \small{$\|u-u_h^m\|_{\infty}$} & order \\ \cline{2-8}
            &   8  & 1.233e-2 &       & 3.306e-1 &       &2.345e-2&	      	\\ %\cline{2-8}
            &  16  & 3.260e-3 & 1.919 & 1.694e-1 & 0.965 &6.765e-3&	1.793 	 \\ %\cline{2-8}
Modified    &  32  & 8.269e-4 & 1.979 & 8.554e-2 & 0.986 &1.775e-3&	1.931 	 \\ %\cline{2-8}
$P_1$-IFEM  &  64  & 2.094e-4 & 1.982 & 4.300e-2 & 0.992 &4.621e-4&	1.941 	 \\ %\cline{2-8}
            & 128  & 5.286e-5 & 1.986 & 2.156e-2 & 0.996 &1.185e-4&	1.964 	 \\ %\cline{2-8}
            & 256  & 1.328e-5 & 1.993 & 1.078e-2 & 0.999 &2.991e-5&	1.986 	 \\ %\cline{2-8}
            & 512  & 3.308e-6 & 2.005 & 5.399e-3 & 0.998 &7.557e-6&	1.985 	 \\ %\cline{2-8}
\hline
\end{tabular}
\begin{center}
\caption{Example \ref{example:cubic} (Cubic curve): $\beta^-=1, ~\beta^+=10$} \label{table:third_10}
\end{center}
\end{table}
\begin{table}
\small
\begin{tabular}{|c|c||c|c||c|c||c|c|}
\hline   & $1/h_x$& \small{$\|u-u_h\|_{0}$} & order & \small{$\|u-u_h\|_{1,h}$} & order &\small{$\|u-u_h\|_{\infty}$}& order\\
\cline{2-8}
          &   8  & 1.923e-2 &       & 3.530e-1 &       & 5.617e-2&	      	\\ %\cline{2-8}
          &  16  & 4.002e-3 & 2.264 & 1.716e-1 & 1.040 & 1.470e-2&	1.934   \\ %\cline{2-8}
\multirow{2}{*}{$P_1$-IFEM}
          &  32  & 9.196e-4 & 2.122 & 8.453e-2 & 1.022 & 3.854e-3&	1.932  \\ %\cline{2-8}
          &  64  & 2.291e-4 & 2.005 & 4.221e-2 & 1.002 & 1.288e-3&	1.582  \\ %\cline{2-8}
          & 128  & 5.408e-5 & 2.083 & 2.105e-2 & 1.004 & 2.836e-4&	2.183  \\ %\cline{2-8}
          & 256  & 1.337e-5 & 2.016 & 1.056e-2 & 0.995 & 1.159e-4&	1.291  \\ %\cline{2-8}
          & 512  & 3.336e-6 & 2.002 & 5.304e-3 & 0.994 & 5.258e-5&	1.141  \\ % \cline{2-8}
     %     & Order&          & 2.072 &          & 1.008 &          & 1.743 \\
\hline  \hline & $1/h_x$& \small{$\hspace{2pt}\|u-u_h^m\|_{0}\hspace{1pt}$} & order & \small{$\hspace{1pt}\|u-u_h^m\|_{1,h}$} & order & \small{$\|u-u_h^m\|_{\infty}$} & order \\ \cline{2-8}
          &   8  & 1.266e-2 &       & 3.216e-1 &       & 2.470e-2&	        \\ %\cline{2-8}
          &  16  & 3.205e-3 & 1.982 & 1.643e-1 & 0.969 & 6.836e-3&	1.854    \\ %\cline{2-8}
Modified  &  32  & 8.163e-4 & 1.973 & 8.293e-2 & 0.986 & 1.784e-3&	1.938    \\ %\cline{2-8}
$P_1$-IFEM&  64  & 2.068e-4 & 1.981 & 4.172e-2 & 0.991 & 4.642e-4&	1.943    \\ %\cline{2-8}
          & 128  & 5.199e-5 & 1.992 & 2.093e-2 & 0.996 & 1.185e-4&	1.970    \\ %\cline{2-8}
          & 256  & 1.302e-5 & 1.998 & 1.048e-2 & 0.998 & 3.009e-5&	1.977    \\ %\cline{2-8}
          & 512  & 3.259e-6 & 1.998 & 5.243e-3 & 0.999 & 7.564e-6&	1.992    \\ % \cline{2-8}
 %         & Order&          & 1.987 &          & 0.991 &          & 2.036 \\
\hline
\end{tabular}
\begin{center}
\caption{Example \ref{example:cubic} (Cubic curve) : $\beta^-=1, ~\beta^+=1000$} \label{table:third_1000}
\end{center}
\end{table}

%From Tables \ref{table:third_10} and \ref{table:third_1000}, we see
%the modified method performs better than the conventional  $P_1$-IFEM when $ \beta^+/\beta^-=10$.

The comparison with error surfaces in Figure \ref{fig_cubic_surf}. shows that modified method gives much more  accurate results
than the original $P_1$-IFEM when $ \beta^+/\beta^-=10$ and $1/h_x = 128$.
The smaller the mesh, the more accurate results the modified method shows.

Table \ref{table:third_10} shows the comparison of errors between the two methods when $ \beta^+/\beta^-=10$.
We can see the original $P_1$-IFEM has suboptimal convergence as the  grids are refined ($1/h_x = 256$ and $512$).
However, the modified method shows a robust order of convergence for all grids.

On the other hands, Table \ref{table:third_1000} shows both methods has an optimal convergence in $L^2$ and $H^1$ norms when $ \beta^+/\beta^-=1000$.

\begin{remark}
  {
Comparing Tables 1 and 2, we see the original $P_1$-IFEM behaves better when $\beta^+/\beta^-=1000$ than $ \beta^+/\beta^-=10$.
 This is a common phenomenon for all the examples we tested. This seems to contradict the usual behavior of standard FEM. We guess the reason is that the large ratio between the
coefficients masks the discontinuity of basis functions.
 Figure \ref{comparison_P1_basis} shows the behavior of $P_1$-IFEM basis between $ \beta^+/\beta^-=10$ and $ \beta^+/\beta^-=1000$. When $ \beta^+/\beta^-=10$, the gap between adjacent elements is conspicuous. However, when $ \beta^+/\beta^-=1000$, the gap is almost invisible.}
\end{remark}

%\begin{figure}
%    \centering
%    \begin{subfigure}[b]{0.45\textwidth}
%        \includegraphics[width=6.0cm]{P1_IFEM_basis_1vs10}
%        \caption{$ \beta^+/\beta^-=10$}
%    \end{subfigure}
%    \begin{subfigure}[b]{0.45\textwidth}
%        \includegraphics[width=6.0cm]{P1_IFEM_basis_1vs1000}
%        \caption{$\beta^+/\beta^-=1000$}
%    \end{subfigure}
%\caption{Comparison of $P_1$-IFEM basis with different $\beta^+/\beta^-$}  \label{comparison_P1_basis}
%\end{figure}

\begin{example} [Sharp corner] \label{example:sharp}
In this example, we consider an interface with a sharp corner having interior angle $2\theta$.
Let  $\Gamma$ be the zero set of $L(x,y) = -y^2 + ((x-0.6) \tan \theta )^2(x+0.4)$ for $x\le 0.6.$
We test the case with  $\theta=45^\circ$  and  $ \beta^+/\beta^-=10$ and $1/10$.
The exact solution is $u = {L(x,y)}/{\beta} $.
\end{example}

%\begin{figure}[hbt]
%\centerline{ \includegraphics[width=8cm]{Sharp_corner_45degree_Level6}}
%\begin{adjustwidth}{0.8cm}{}
%\caption[c]
%{\begin{minipage}[t]{.8\linewidth}Solution $u_h$ for Example \ref{example:sharp} (Sharp Corner) \\($\beta^-=1$, $\beta^+=10$, $\theta = 45^{\circ}$, $   {1/h_x = 32}$) \end{minipage}}\label{fig_sol_sharp}
%\end{adjustwidth}
%\end{figure}

%\begin{figure}
%\begin{center}
%    \centering
%    \begin{subfigure}[b]{0.45\textwidth}
%        \includegraphics[width=6.0cm]{Sharp_corner_45degree_10_1_original}
%        \caption{$P_1$-IFEM}
%    \end{subfigure}
%    \begin{subfigure}[b]{0.45\textwidth}
%        \includegraphics[width=6.0cm]{Sharp_corner_45degree_10_1_modified}
%        \caption{Modified $P_1$-IFEM}
%    \end{subfigure}
%    \begin{adjustwidth}{0.8cm}{}
%    \caption[c]
%      {\begin{minipage}[t]{.8\linewidth}Error Surfaces for Example \ref{example:sharp} (sharp corner) \\(  {$\beta^-=1$, $\beta^+=10$, $\theta = 45^{\circ}$, $1/h_x = 128$}) \end{minipage}}
%      \end{adjustwidth}
%%{ \small Error Surfaces for Example \ref{example:sharp} (  {$\beta^-=1, \beta^+=10, \theta = 45^{\circ}$}) }
%\end{center}
%\end{figure}
\begin{table}
\small
\begin{tabular}{|c|c||c|c||c|c||c|c|}
\hline   & $1/h_x$& \small{$\|u-u_h\|_{0}$} & order & \small{$\|u-u_h\|_{1,h}$} & order &\small{$\|u-u_h\|_{\infty}$}& order\\
\cline{2-8}
          &   8  & 3.359e-3 &       & 7.958e-2 &       & 1.036e-2&	       	 \\ %\cline{2-8}
          &  16  & 9.014e-4 & 1.898 & 4.185e-3 & 0.927 & 4.118e-3&	1.332    \\ %\cline{2-8}
\multirow{2}{*}{$P_1$-IFEM}
          &  32  & 2.219e-4 & 2.022 & 2.161e-3 & 0.954 & 1.958e-3&	1.073    \\ %\cline{2-8}
          &  64  & 5.686e-5 & 1.965 & 1.197e-3 & 0.852 & 9.568e-4&	1.033    \\ %\cline{2-8}
          & 128  & 1.463e-5 & 1.958 & 6.573e-3 & 0.865 & 5.063e-4&	0.918    \\ %\cline{2-8}
          & 256  & 6.070e-6 & 1.269 & 3.967e-3 & 0.728 & 2.462e-4&	1.040    \\ %\cline{2-8}
          & 512  & 2.942e-6 & 1.045 & 2.439e-3 & 0.702 & 1.241e-4&	0.988    \\ %\cline{2-8}
     %     & Order&          & 1.744 &          & 0.843 &          & 1.090 \\
\hline  \hline & $1/h_x$& \small{$\hspace{2pt}\|u-u_h^m\|_{0}\hspace{1pt}$} & order & \small{$\hspace{1pt}\|u-u_h^m\|_{1,h}$} & order & \small{$\|u-u_h^m\|_{\infty}$} & order \\ \cline{2-8}
          &   8  & 3.056e-3 &       & 7.817e-2 &       & 9.005e-3&	       	\\ %\cline{2-8}
          &  16  & 7.441e-4 & 2.038 & 3.956e-2 & 0.983 & 2.316e-3&	1.959   \\ %\cline{2-8}
Modified  &  32  & 1.930e-4 & 1.947 & 1.990e-2 & 0.991 & 6.221e-4&	1.896   \\ %\cline{2-8}
$P_1$-IFEM&  64  & 4.716e-5 & 2.033 & 1.000e-2 & 0.993 & 1.608e-4&	1.952   \\ %\cline{2-8}
          & 128  & 1.216e-5 & 1.956 & 5.015e-3 & 0.996 & 4.090e-5&	1.975   \\ %\cline{2-8}
          & 256  & 3.010e-6 & 2.014 & 2.510e-3 & 0.999 & 1.031e-5&	1.989   \\ %\cline{2-8}
          & 512  & 7.621e-7 & 1.982 & 1.256e-3 & 0.999 & 2.633e-6&	1.968   \\ %\cline{2-8}
        %  & Order&          & 1.993 &          & 0.994 &          & 1.935 \\
\hline
\end{tabular}
\begin{center}
\caption{ Example \ref{example:sharp} (Sharp corner) : $\theta= 45^{\circ}, ~ \beta^-=1, ~\beta^+=10$} \label{table:sharp_edge45_10}
\end{center}
\end{table}

\begin{table}
\small
\begin{tabular}{|c|c||c|c||c|c||c|c|}
\hline   & $1/h_x$& \small{$\|u-u_h\|_{0}$} & order & \small{$\|u-u_h\|_{1,h}$} & order &\small{$\|u-u_h\|_{\infty}$}& order \\
\cline{2-8}
          &   8  & 1.238e-2 &       & 3.013e-1 &       & 1.613e-2&	       \\ %\cline{2-8}
          &  16  & 3.159e-3 & 1.971 & 1.513e-1 & 0.994 & 4.327e-3&	1.899  \\ %\cline{2-8}
\multirow{2}{*}{$P_1$-IFEM}
          &  32  & 7.949e-4 & 1.991 & 7.572e-2 & 0.998 & 1.174e-3&	1.882  \\ %\cline{2-8}
          &  64  & 2.030e-4 & 1.969 & 3.821e-2 & 0.987 & 7.475e-4&	0.651  \\ %\cline{2-8}
          & 128  & 5.366e-5 & 1.920 & 1.933e-2 & 0.983 & 4.704e-4&	0.668  \\ %\cline{2-8}
          & 256  & 1.528e-5 & 1.812 & 9.919e-3 & 0.963 & 2.452e-4&	0.940  \\ %\cline{2-8}
          & 512  & 4.898e-6 & 1.642 & 5.155e-3 & 0.944 & 1.199e-4&	1.033  \\ %\cline{2-8}
     %     & Order&          & 1.744 &          & 0.843 &          & 1.090 \\
\hline  \hline & $1/h_x$& \small{$\hspace{2pt}\|u-u_h^m\|_{0}\hspace{1pt}$} & order & \small{$\hspace{1pt}\|u-u_h^m\|_{1,h}$} & order & \small{$\|u-u_h^m\|_{\infty}$} & order \\ \cline{2-8}
          &   8  & 1.238e-2 &       & 3.010e-1 &       & 1.610e-2&	       \\ %\cline{2-8}
          &  16  & 3.094e-3 & 2.000 & 1.507e-1 & 0.998 & 4.107e-3&	1.971  \\ %\cline{2-8}
Modified  &  32  & 7.787e-4 & 1.990 & 7.543e-2 & 0.999 & 1.037e-3&	1.986  \\ %\cline{2-8}
$P_1$-IFEM&  64  & 1.947e-4 & 2.000 & 3.773e-2 & 0.999 & 2.605e-4&	1.993  \\ %\cline{2-8}
          & 128  & 4.876e-5 & 1.998 & 1.887e-2 & 1.000 & 6.528e-5&	1.997  \\ %\cline{2-8}
          & 256  & 1.219e-5 & 2.000 & 9.435e-3 & 1.000 & 1.634e-5&	1.998  \\ %\cline{2-8}
          & 512  & 3.051e-6 & 1.998 & 4.718e-3 & 1.000 & 4.087e-6&	1.999  \\ %\cline{2-8}
        %  & Order&          & 1.993 &          & 0.994 &          & 1.935 \\
\hline
\end{tabular}
\begin{center}
\caption{ Example \ref{example:sharp} (Sharp corner) : $\theta= 45^{\circ}, ~ \beta^-=10, ~\beta^+=1$} \label{table:sharp_edge45_01}
\end{center}
\end{table}

This example is not covered by analysis of this work because the problem has low regularity at the interface corner.
However, we see that the modified method works better; See the Table \ref{table:sharp_edge45_10}.

\begin{remark}
 We have also computed other cases such as $\beta^+/\beta^-=1000, 0.001$ with various angles.
 The results of our scheme are always optimal while the unmodified $P_1$ immersed method deteriorates for some cases.
\end{remark}

\begin{example}[Variable coefficient] \label{example:variable}
Finally, we consider the case with variable coefficient.
  {The 0-set of function $L(x,y) = x^2/(0.9)^2+y^2/(0.5)^2-1.0$ is used in this example as the interface.}
%We take the level set of $L(x,y) = x^2/(0.9)^2+y^2/(0.5)^2-1.0$ as an  interface.
The exact solution is
$u=L(x,y)/\beta(x,y)$ where
\begin{eqnarray*}
 \beta(x,y) = \left\{%
\begin{array}{ll}
      (x^2+y^2-1)^2  & \textrm{on $\Omega^-$,} \\
      1 & \textrm{on $\Omega^+.$} \\
\end{array}%
\right.
\end{eqnarray*}
In this case, both methods show an optimal order of convergence in $H^1$-norm.  But the modified method
performs much better in $L^\infty$-norm; See Table \ref{table:ellipse_variable_coeff}.
\end{example}

%\begin{figure}[htb]
%\centerline{ \includegraphics[width=8cm]{variable_coeff_level6}}
%\caption{Solution $u_h$ for Example \ref{example:variable} (Variable coefficient)}
%\end{figure}

%\begin{figure}[htb]
%    \begin{subfigure}[b]{0.45\textwidth}
%        \includegraphics[width=6.0cm]{Ellipse_original}
%        \caption{$P_1$-IFEM}
%    \end{subfigure}
%    \begin{subfigure}[b]{0.45\textwidth}
%        \includegraphics[width=6.0cm]{Ellipse_modified}
%        \caption{Modified $P_1$-IFEM}
%    \end{subfigure}
%\begin{adjustwidth}{0.0cm}{}
%\caption[c]
%{\begin{minipage}[t]{.8\linewidth}Error Surface for Example \ref{example:variable} (Variable coefficient) \\
%  {($ \beta^- = (x^2+y^2-1)^2$, $\beta^+ = 1$, $1/h_x = 128$)}
%\end{minipage}}\label{fig_err_surf_variable}
%\end{adjustwidth}
%\end{figure}

\begin{table}
\small
\begin{tabular}{|c|c||c|c||c|c||c|c|}
\hline   & $1/h_x$& \small{$\|u-u_h\|_{0}$} & order & \small{$\|u-u_h\|_{1,h}$} & order &\small{$\|u-u_h\|_{\infty}$}& order\\
\cline{2-8}
          &   8  & 8.550e-2 &       & 1.585e-0 &       & 2.415e-1&	        \\ %\cline{2-8}
          &  16  & 2.931e-2 & 1.544 & 9.840e-1 & 0.688 & 1.025e-1&	1.237   \\ %\cline{2-8}
\multirow{2}{*}{$P_1$-IFEM}
          &  32  & 7.954e-3 & 1.882 & 5.538e-1 & 0.829 & 4.174e-2&	1.295   \\ %\cline{2-8}
          &  64  & 2.002e-3 & 1.990 & 3.033e-1 & 0.869 & 1.568e-2&	1.413   \\ %\cline{2-8}
          & 128  & 4.825e-4 & 2.053 & 1.665e-1 & 0.865 & 8.471e-3&	0.888   \\ %\cline{2-8}
          & 256  & 1.206e-4 & 2.000 & 8.948e-2 & 0.896 & 4.393e-3&	0.947   \\ %\cline{2-8}
          & 512  & 3.461e-5 & 1.802 & 5.063e-2 & 0.822 & 2.132e-3&	1.043   \\ %\cline{2-8}
       %   & Order&          & 1.918 &          & 0.841 &          & 1.295 \\
\hline  \hline & $1/h_x$& \small{$\hspace{2pt}\|u-u_h^m\|_{0}\hspace{1pt}$} & order & \small{$\hspace{1pt}\|u-u_h^m\|_{1,h}$} & order & \small{$\|u-u_h^m\|_{\infty}$} & order \\ \cline{2-8}
          &   8  & 8.652e-2 &       & 1.572e-0 &       & 2.150e-1&	        \\ %\cline{2-8}
          &  16  & 2.867e-2 & 1.593 & 9.704e-1 & 0.696 & 9.448e-2&	1.187   \\ %\cline{2-8}
Modified  &  32  & 8.049e-3 & 1.833 & 5.368e-1 & 0.854 & 3.656e-2&	1.370   \\ %\cline{2-8}
$P_1$-IFEM&  64  & 2.195e-3 & 1.874 & 2.889e-1 & 0.894 & 1.097e-2&	1.736   \\ %\cline{2-8}
          & 128  & 5.585e-4 & 1.975 & 1.485e-1 & 0.959 & 3.055e-3&	1.845   \\ %\cline{2-8}
          & 256  & 1.437e-4 & 1.958 & 7.550e-2 & 0.976 & 8.386e-4&	1.865   \\ %\cline{2-8}
          & 512  & 3.649e-5 & 1.978 & 3.809e-2 & 0.987 & 2.209e-4&	1.925   \\ %\cline{2-8}
       %   & Order&          & 1.884 &          & 0.904 &          & 1.884 \\
\hline
\end{tabular}
\begin{center}
\caption{ Example \ref{example:variable} (Variable coefficient)} \label{table:ellipse_variable_coeff}
\end{center}
\end{table}

\section{Conclusion}
   {
We introduced a modified IFEM for solving elliptic interface problems.
By adding the line integral terms similar to the DG methods, we overcome the suboptimal behavior of the original IFEM
proposed in \cite{Li2003}, \cite{Li2004}.
(The computational result there seemed to show optimal order. However, more
numerical experiments show the original $P_1$-IFEM is not optimal for some problems. The proof in
\cite{Ch-Kw-We} seems incorrect. However, our modified scheme is always robust for all problems tested including unreported ones).
  The optimal convergence rates in $H^1$ and $L^2$ norms are shown by a similar technique as in DG methods; the modified IFEM is  consistent, the coercivity and boundedness of the bilinear form hold.
  Several numerical tests show the errors are $O(h)$, $O(h^2)$ order in respective norms.
Although no proof is given, we also obtain $O(h^2)$ order in $L^\infty$ norm.
}

  { Some of the limitation of our scheme might be these: for problems with  sharp interface, one has
to arrange the grids so that  the cusp point is located at a vertex of an element,
 for problems with highly oscillating interface, further refinement are necessary to apply the IFEM.}

  { We  now comment on the  computational aspects:  The matrix structure are exactly the same as usual $P_1$-FEM, i.e., 5-point stencil; the number of unknowns are also the same. When $\eps=-1$, the scheme becomes symmetric. The assembly of stiffness matrix requires slightly more time than the unmodified $P_1$-IFEM, but the time for iterative solver such as conjugate gradient is almost the same.}

  {An obvious advantage of (both) IFEM is that we can use fast solver such as multigrid methods since we can use uniform meshes.}

%IFEM can be applied to any regular FEM meshes as long as the interface is smooth.

   {
 Future works related to this topic are:
 \begin{enumerate} \item Local refinement near singularity.
\item Problems with nonhomogenous jumps, tensor coefficients, etc.
\item $Q_1$-IFEM for rectangular elements.
\item 3-dimensional problems.
\item Problems with a moving interface.
\item  Two phase Stokes/Navier-Stokes problems.
\item Development of fast solver such as multigrid methods.
\end{enumerate}
}

\end{document}